\documentclass[10pt, a4paper, reqno, english]{amsart}

\usepackage[bookmarks=false]{hyperref}
\usepackage{amssymb}

\newcommand\RR{\mathbb{R}}
\newcommand\PP{\mathbb{P}}
\newcommand\QQ{\mathbb{Q}}

\newcommand\OO{\mathcal{O}}
\newcommand\AAA{\mathcal{A}}

\newcommand\cl[1]{[#1]}
\newcommand\ideals{\mathcal{I}_{K,S}}
\newcommand\sint{\OO_{K,S}}
\newcommand\sunit{\OO_{K,S}^\times}
\newcommand\sintnmodunits{\OO_{K,S}^n/{\sim}}

\newcommand\unmq{u(n,m;q)}
\newcommand\vnmq{V(n,m;q)}
\newcommand\vnmqdist{V^*(n,m;q)}
\newcommand\balpha{\boldsymbol\alpha}
\newcommand\bbeta{\boldsymbol\beta}
\newcommand\Sn{S_n}

\DeclareMathOperator\Reg{Reg}
\DeclareMathOperator\id{id}

\newtheorem{theorem}{Theorem}
\newtheorem{lemma}[theorem]{Lemma}

\newtheorem{problem}[theorem]{Problem}
\theoremstyle{definition}

\begin{document}

\setcounter{tocdepth}{1}

\title{On sums of S-integers of bounded norm}
\dedicatory{Dedicated to Wolfgang M. Schmidt on the occasion of his 80th birthday}

\author{Christopher Frei}
\author{Robert Tichy}
\address{Institut f\"ur Mathematik A, Technische Universit\"at Graz,
  Steyrergasse 30, 8010 Graz, Austria}
\email{frei@math.tugraz.at}
\email{tichy@tugraz.at}

\author{Volker Ziegler}
\address{Johann Radon Institute for Computational and Applied Mathematics (RICAM)\\
Austrian Academy of Sciences\\
Altenbergerstr. 69\\
A-4040 Linz, Austria}
\email{volker.ziegler\char'100ricam.oeaw.ac.at}

\date{July 29, 2013}

\begin{abstract}
 We prove an asymptotic formula for the number of $S$-integers in a number
 field $K$ that can be represented by a sum of $n$ $S$-integers of bounded norm.
\end{abstract}


\maketitle

\section{Introduction}
A (weighted) $S$-unit equation is an equation over a field $K$ of the form
\[a_1 x_1 +\cdots + a_n x_n=0,\] where $a_i \in K$ are fixed, $x_i \in \Gamma$
and $\Gamma$ is a finitely generated subgroup of $K^*$. Usually $K$ is a number
field and $\Gamma$ a group of $S$-units.

The study of $S$-unit equations for their own interest and also in view of
applications has a long history. Siegel and Mahler
\cite{Mahler:1933, Siegel:1929} already considered $S$-unit equations in order
to prove the finiteness of rational points on certain curves.  Nowadays a
standard tool to investigate $S$-unit equations is Wolfgang Schmidt's famous
Subspace Theorem \cite{Schmidt:1971} or one of its generalizations due to
Evertse, Ferretti or Schlickewei
\cite{Evertse:2013,Evertse:2002a,Schlickewei:1976,Schlickewei:1977}.

The theory of $S$-unit equations was applied, for example, to count the number
of solutions to norm form equations (see
e.g. \cite{Evertse:1997,Schlickewei:1977}) or to study arithmetic properties of
recurrence sequences (see e.g.  \cite{Evertse:2002,Schmidt:2003}).

In this paper we consider applications of $S$-unit equations to the unit sum
number problem, i.e.\ the question which number fields have the property that
each algebraic integer is a sum of units. In the last decade this problem has
been studied by several authors; we refer to \cite{Barroero:2011} for an
overview. We are interested in questions related to a problem posed by Jarden
and Narkiewicz \cite[Problem C]{Jarden2007}:

\begin{problem}
  Let $K$ be an algebraic number field. Obtain an asymptotic formula for the
  number $N_n(x)$ of positive integers $m\leq x$ which are sums of at most $n$
  units of the ring of integers of $K$.
\end{problem}

Variants of this problem were studied by Filipin, Fuchs, Tichy and Ziegler
\cite{Filipin2008B,Fuchs2009}. Using a result of Everest \cite{Everest1990},
they found an asymptotic formula for the number of non-associated integers in
$K$ of bounded norm that are sums of exactly $n$ units; see also
\cite{Everest1999,Gyory1990}. The aim of this article is to close a gap in
\cite{Filipin2008B,Fuchs2009} and to further generalize the results.

For $n$, $m$ fixed, we consider integers in $K$ that are sums of exactly $n$
integers in $K$ of norm at most $m$. It follows from \cite{Hajdu2007} that not
every integer in $K$ can be written as such a sum. The natural next step is to
show a quantitative result in the style of \cite{Fuchs2009}. Everest's theorem
allows us to achieve this without much additional effort.

Before we state our main result we fix
some notation. Let $K$ be a number field, $S$ a finite set of places of $K$
containing all Archimedean places, $s := |S|-1$ (where $|S|$ denotes the
cardinality of the set $S$), and $\sint$ the ring
of $S$-integers of $K$. We say that two $S$-integers $\alpha$, $\beta
\in \sint\smallsetminus\{0\}$ are \emph{associated}, $\alpha \sim\beta$, if $\alpha/\beta
\in \sunit$. This is an equivalence relation, and we denote the equivalence class of $\alpha$ by
$\cl{\alpha}$.  Let $\unmq$ be the number of equivalence classes $[\alpha]$ of nonzero $S$-integers such that
\begin{equation*}
  N_S(\alpha) := \prod_{v \in S}|\alpha|_v \leq q\quad\text{ and }\quad\alpha = \sum_{i=1}^n\alpha_i\text,
\end{equation*}
where $\alpha_i \in \sint\smallsetminus\{0\}$ such that $N_S(\alpha_i)\leq
m$ and 
\begin{equation}\label{eq:no_subsum_vanishes}
\text{ no subsum of $\alpha_1 + \cdots + \alpha_n$ vanishes.}
\end{equation}
(The absolute values are normalized by $|a|_v = |a|_w^{[K_v :
  \QQ_w]}$ for $a \in \QQ$, where $w$ is the place of $\QQ$ below $v$ and
$|\cdot|_w$ is the usual $w$-adic absolute value on $\QQ$. By the product formula, $N_S(\alpha)$
 depends only on $[\alpha]$.)

We write $\ideals(m)$ for the set of all nonzero principal ideals $\AAA$ of
$\sint$ of norm $[\sint : \AAA]\leq m$. Moreover, we define the constant
$c_{n,s}$ to be the $ns$-dimensional Lebesgue measure of 
\[\{(x_{11},\ldots,x_{ns})\in\RR^{ns} \mid g(x_{11},\ldots,x_{ns})<1\},\]
where
\[g(x_{11},\ldots,x_{ns})=\sum_{i=1}^s \max\{0,x_{1i},\ldots,x_{ni}\}
+\max\left\{0,-\sum_{i=1}^s x_{1i},\ldots,-\sum_{i=1}^s x_{ni}\right\}.\]
This positive constant is the same as in \cite{Fuchs2009}. It satisfies the
inequalities
\begin{equation*}
  \frac{2^{ns}}{(ns)!}< c_{n,s} < 2^{ns},
\end{equation*}
and its exact values are known for $n = 1$ or
$s\leq 2$ \cite{Barroero:2011, Fuchs2009}.

\begin{theorem}\label{thm:main}
With the above notation the following asymptotic formula holds as $q \to \infty$:
\begin{equation*}
  \unmq =
  |\ideals(m)|^n\frac{c_{n-1,s}}{n!}\left(\frac{\omega_K(\log
      q)^s}{\Reg_{K,S}}\right)^{n-1}+O((\log q)^{(n-1)s-1})\text, 
\end{equation*}
where $\omega_K$ is the number of roots of unity in $K$ and $\Reg_{K,S}$
is the $S$-regulator of $K$.
\end{theorem}

We note that the case $m=1$ of Theorem \ref{thm:main} is just
\cite[Theorem 1]{Fuchs2009}.

\section{Proof of Theorem \ref{thm:main}}

For each principal ideal $\AAA \in \ideals(m)$, we choose a fixed generator
$g_\AAA \in \sint$, that is, $g_\AAA\sint = \AAA$. Then $N_S(g_\AAA) = [\sint :
\AAA]$. Let $G(m)$ be the set of all $g_\AAA$ with $\AAA \in \ideals(m)$. We
extend the equivalence relation $\sim$ to $n$-tuples in $\sint^n$ by
\begin{equation*}
  (\alpha_1, \ldots, \alpha_n) \sim (\beta_1, \ldots, \beta_n) :\Leftrightarrow (\alpha_1, \ldots, \alpha_n) = 
  \lambda \cdot(\beta_1, \ldots, \beta_n) \text{ for some }\lambda\in\sunit, 
\end{equation*}
and write $\balpha = (\alpha_1:\cdots:\alpha_n)$ for the equivalence
class of $(\alpha_1, \ldots, \alpha_n)$. For $\balpha \in
\sintnmodunits$, the expression
\begin{equation*}
N_{S}(\alpha_1 + \cdots + \alpha_n) = \prod_{v \in S}|\alpha_1 + \cdots + \alpha_n|_v\text,
\end{equation*}
is well defined. Each $\alpha \in \sint\smallsetminus\{0\}$ with $N_S(\alpha) \leq m$ can be
written uniquely as $\alpha = g(\alpha)\epsilon$, where $g(\alpha)$ is the
generator in $G(m)$ of the principal ideal $\alpha\sint$ and $\epsilon :=
\alpha/g(\alpha) \in \sunit$.

Our main tool to prove Theorem \ref{thm:main} is a result due to Everest
\cite[Theorem]{Everest1990}. We write $\PP^{n-1}(\sunit)$ for the set of
equivalence classes $(\epsilon_1 : \cdots : \epsilon_n)$ with $\epsilon_i \in
\sunit$ for all $i \in \{1, \ldots, n\}$.

\begin{theorem}[Everest \cite{Everest1990}]\label{thm:Everest}
 For fixed $c=(c_1,\ldots,c_n)\in (K\smallsetminus \{0\})^n$, let $w_c(n;q)$ be the
 number of all  $(\epsilon_1 : \cdots : \epsilon_n)\in \PP^{n-1}(\sunit)$ such
 that

\begin{equation*}
   N_S(c_1\epsilon_1 +
   \cdots + c_n\epsilon_n)\leq q
\end{equation*}
and
\begin{equation}\label{eq:no_subsum_vanishes_eps}
  \text{no subsum of $c_1\epsilon_1 +
    \cdots + c_n\epsilon_n$ vanishes.} 
\end{equation}
 Then, as $q\to\infty$,
 \[w_c(n;q)=c_{n-1,s}\left(\frac{\omega_K(\log q)^s}{\Reg_{K,S}}\right)^{n-1}+O((\log q)^{(n-1)s-1}).\]
\end{theorem}

In view of Theorem \ref{thm:main} we are interested in the set
\begin{align*}
  \vnmq &:= \{\balpha \in\sintnmodunits \mid N_S(\alpha_1 + \cdots + \alpha_n)\leq
  q\text{, } N_S(\alpha_i) \leq m\text{, }\eqref{eq:no_subsum_vanishes}\}.
\end{align*}

\begin{lemma}\label{lem:count_tuples}
  We have, as $q \to \infty$,
  \begin{equation*}
    |\vnmq| = |\ideals(m)|^n c_{n-1,s} \left(\frac{\omega_K(\log
      q)^s}{\Reg_{K,S}}\right)^{n-1}+O((\log q)^{(n-1)s-1}).
  \end{equation*}
\end{lemma}

\begin{proof}
  Condition \eqref{eq:no_subsum_vanishes} implies that
  $\alpha_i \neq 0$ holds for all $i$. With $c_i := g(\alpha_i)$ and $\epsilon_i := \alpha_i/c_i$, we
  can write $|\vnmq|$ as
  \begin{align*}
    \sum_{(c_1, \ldots, c_n) \in G(m)^n}|\{(\epsilon_1 : \cdots : \epsilon_n)\in \PP^{n-1}(\sunit) \mid N_S(c_1\epsilon_1 +
    \cdots + c_n\epsilon_n)\leq q\text{,
    }\eqref{eq:no_subsum_vanishes_eps}\}|.
  \end{align*}
  The lemma is an immediate consequence of Theorem \ref{thm:Everest}.
\end{proof}

Let us introduce another equivalence relation $\sim_P$ on $\sintnmodunits$: We say that two elements
$\balpha, \bbeta \in \sintnmodunits$ are equivalent if $\bbeta$ arises from
$\balpha$ by a permutation of the coordinates. By $[\balpha]_P$ we denote the equivalence class of $\balpha \in \sintnmodunits$ with respect to 
$\sim_P$.
Note that each equivalence class with respect to $\sim_P$ has at most $n!$ elements.

Let $\vnmqdist$ be the set of all $\balpha \in \vnmq$ whose equivalence class $[\balpha]_P$
has exactly $n!$ elements.

\begin{lemma}\label{lem:n!_elements}
  We have, as $q \to \infty$,
  \begin{equation*}
     |\vnmqdist| = |\ideals(m)|^n c_{n-1,s} \left(\frac{\omega_K(\log
      q)^s}{\Reg_{K,S}}\right)^{n-1}+O((\log q)^{(n-1)s-1}).
  \end{equation*}
\end{lemma}

\begin{proof}
  If the equivalence class of $\balpha$ has less than $n!$ elements then there
  is a permutation $\pi\in\Sn$, $\pi \neq \id$, such that $(\alpha_{\pi(1)} :
  \cdots : \alpha_{\pi(n)}) = (\alpha_1 : \cdots : \alpha_n)$, that is,
  \begin{equation*}
    (\alpha_{\pi(1)}, \ldots, \alpha_{\pi(n)}) = (\lambda\alpha_1, \ldots, \lambda\alpha_n),
  \end{equation*}
  for some $\lambda \in \sunit$. Assume that, say, $\pi(2) = 1$. Since
  $\alpha_2 = \alpha_{\pi^{n!}(2)} = \lambda^{n!}\alpha_2$ and $\alpha_2 \neq
  0$, we see that $\lambda$ is a $(n!)$-th root of unity. We have $\alpha_1 =
  \lambda\alpha_2 = g(\alpha_2)\lambda\epsilon_2$. In particular,
  \eqref{eq:no_subsum_vanishes} implies that $g(\alpha_2)(\lambda+1)\neq 0$, so the number
  of such $\balpha$ in $\vnmq$ for which there exists such a permutation
  $\pi$ is bounded by 
  \begin{equation*}
    \sum_{\substack{(c_2, \ldots, c_n)\in
        G(m)^{n-1}\\\lambda^{n!}=1\text{, }\lambda\neq -1}} \!\!\!\!\!\!\!\! |\{(\epsilon_2:\cdots:\epsilon_n)\in\PP^{n-2}(\sunit)\mid
    N_S((c_2\lambda + c_2)\epsilon_2+\cdots+c_n\epsilon_n)\leq q\text{, }~(\ref{eq:no_subsum_vanishes_eps})\}|\text.
  \end{equation*}
  This is $\ll (\log q)^{(n-2)s}$ by Theorem \ref{thm:Everest}. The above argument
  has to be carried out for all ${n \choose 2}$ pairs $1 \leq i<j \leq
  n$. Thus, the number of equivalence classes $[\balpha]_P$ with less than $n!$
  elements is $\ll (\log q)^{(n-2)s}$, which proves the result.
\end{proof}

Clearly, for $\balpha \in \vnmq$, the equivalence class $[\alpha_1 + \cdots + \alpha_n]$ depends only on the equivalence class $[\balpha]_P$. We show that for 
most of the $\balpha \in \vnmqdist$, there is no $\bbeta \in \vnmqdist\smallsetminus [\balpha]_P$ with $[\beta_1 + \cdots + \beta_n] = [\alpha_1 + \cdots + 
\alpha_n]$. 

\begin{lemma}\label{lem:almost_unique}
  The number of $\balpha \in \vnmqdist$ for which there is $\bbeta \in \vnmqdist\smallsetminus[\balpha]_P$ with $[\alpha_1 + \cdots + \alpha_n] = [\beta_1 + 
\cdots + \beta_n]$ is $\ll (\log q)^{(n-2)s}$ as $q\to \infty$.
\end{lemma}

\begin{proof}
  If there is such a $\bbeta$ then we can write $\balpha = (\alpha_1 : \cdots :
  \alpha_n)$, $\bbeta = (\beta_1 : \cdots : \beta_n)$, with
  \begin{equation*}
    \alpha_1 + \cdots + \alpha_n - \beta_1 - \cdots - \beta_n = 0.
  \end{equation*}
  Since $\balpha$, $\bbeta \in \vnmqdist$, all $\alpha_i$ (resp. all $\beta_i$)
  are pairwise distinct. Since $\bbeta \notin [\balpha]_P$,
 \begin{equation}\label{eq:sets_different}
   \{\alpha_1, \ldots, \alpha_n\} \neq \{\beta_1, \ldots, \beta_n\}.
 \end{equation}
 Due to \eqref{eq:no_subsum_vanishes} and \eqref{eq:sets_different}, there exist subsets $I, J \subseteq \{1, \ldots, n\}$ with $|I| \geq 2$, such that
 \begin{equation}\label{eq:no_subsum}
   \sum_{i \in I}\alpha_i - \sum_{j \in J}\beta_j = 0
 \end{equation}
 and no proper subsum of \eqref{eq:no_subsum} vanishes. Assume that $I = \{1, \ldots, k\}$, $J = \{1, \ldots, l\}$, for some $k \in \{2, \ldots, n\}$, $l \in 
\{1, \ldots, n\}$.  We write (uniquely) $\alpha_i = c_i \epsilon_i$, $\beta_i = d_i \delta_i$, with $c_i,d_i \in G(m)$ and $\epsilon_i, \delta_i \in \sunit$. 
Then \eqref{eq:no_subsum} becomes a non-degenerate $S$-unit equation
 \begin{equation}
   \label{eq:s_unit_eq}
   \sum_{i = 1}^kc_i\epsilon_i - \sum_{j = 1}^ld_j\delta_j = 0,
 \end{equation}
 which has only finitely many solutions $(\epsilon_1 : \cdots : \epsilon_k :
 \delta_1 : \cdots : \delta_l) \in \PP^{k+l-1}(\sunit)$. Let
 $\mathcal{E}=\mathcal{E}(c_1, \ldots, c_k, d_1, \ldots, d_l)$ be the finite set of
 all representatives of the form
 \begin{equation*}
   (1, \epsilon_2, \ldots, \epsilon_k)
 \end{equation*}
 that appear in a solution of \eqref{eq:s_unit_eq}. Then
 \begin{equation*}
   \balpha = (c_1\mu : c_2\epsilon_2\mu : \cdots : c_k\epsilon_k\mu : c_{k+1}\epsilon_{k+1} : \cdots : c_n\epsilon_n),
 \end{equation*}
with $c_i \in G(m)$, $(1, \epsilon_2, \ldots, \epsilon_k)\in\mathcal{E}$, and
$\mu, \epsilon_{k+1}, \ldots, \epsilon_n \in \sunit$. By
\eqref{eq:no_subsum_vanishes}, we have
\begin{equation}\label{eq:no_subsum_vanishes_c}
  c_1+c_2\epsilon_2+\cdots+c_k\epsilon_k \neq
  0.
\end{equation}
The number of such $\balpha$ is bounded by the sum over all $(c_1, \ldots, c_n)
\in G(m)^n$ and $(1, \epsilon_2, \ldots, \epsilon_k)\in\mathcal{E}$ with
\eqref{eq:no_subsum_vanishes_c} of the number of all
\begin{equation*}
(\mu : \epsilon_{k+1} : \cdots : \epsilon_n) \in \PP^{n-k}(\sunit)
\end{equation*}
with 
\begin{equation*}
  N_S((c_1+c_2\epsilon_2 + \cdots + c_k\epsilon_k)\mu + c_{k+1}\epsilon_{k+1} + \cdots + c_n\epsilon_n)\leq q,
\end{equation*}
and for which no subsum of $(c_1 + \cdots + c_k\epsilon_k)\mu + c_{k+1}\epsilon_{k+1} + \cdots + c_n\epsilon_n$ vanishes. By Theorem \ref{thm:Everest}, this 
number is $\ll (\log q)^{(n-k)s} \leq (\log q)^{(n-2)s}$.
\end{proof}

Every equivalence class counted by $\unmq$ is of the form $[\alpha_1 + \cdots +
\alpha_n]$, for some $\balpha = (\alpha_1 : \cdots : \alpha_n) \in \vnmq$. By
Lemma \ref{lem:count_tuples} and \ref{lem:n!_elements}, we can restrict
ourselves to equivalence classes $[\balpha]_P$ coming from $\balpha \in \vnmqdist$ without
changing the main term. By Lemma \ref{lem:almost_unique}, the equivalence
class $[\balpha]_P$ of each such $\balpha$ is uniquely defined by the class
$[\alpha_1 + \cdots + \alpha_n]$, with $\ll (\log q)^{(n-2)s}$
exceptions. Hence, we obtain the desired asymptotic formula by counting
equivalence classes $[\balpha]_P$ with $\balpha \in \vnmqdist$. (This argument
was not given in \cite{Fuchs2009}.)

\section*{Acknowledgment}
R.~T. was supported by the Austrian Science Fund (FWF) under the project F5510
(part of the Special Research Program (SFB) ``Quasi-Monte Carlo Methods: Theory
and Applications''). V.~Z. was supported by the Austrian Science Fund (FWF) under the project P~24801-N26.

\bibliographystyle{plain}

\bibliography{sums_of_integers_of_bounded_norm}

\end{document}